\documentclass{article}
\usepackage{amssymb, mathrsfs, graphicx, hyperref, fancyhdr}
\usepackage{amsmath}
\usepackage{amsthm}
\usepackage[arc,all]{xy}
\usepackage{tikz}
\usetikzlibrary{shapes,arrows,positioning}
\usepackage{zref-perpage}
\usepackage[nottoc]{tocbibind}
\theoremstyle{definition}
\newtheorem{example}{Example}%[definition]
\theoremstyle{theorem}
\newtheorem{theorem}{Theorem}%[definition]
\newtheorem{proposition}{Proposition}%[definition]
\newtheorem{corollary}{Corollary}%[definition]
\makeatletter
\title{On general $(\alpha, \beta)$-metrics of weak Landsberg type}
\author{ A. Ala, A. Behzadi\footnote{Corresponding author.}  ~~and  M. Rafiei-Rad }

\begin{document}
\maketitle
\begin{abstract}
In this paper, we study general $(\alpha,\beta)$-metrics which $\alpha$ is a Riemannian metric and $\beta$ is an one-form. We have proven that every weak Landsberg general $(\alpha,\beta)$-metric is a Berwald metric, where $\beta$ is a closed and conformal one-form. This show that there exist no generalized unicorn metric in this class of general $(\alpha,\beta)$-metric. Further, We show that $F$ is a Landsberg general $(\alpha,\beta)$-metric if and only if it is weak Landsberg general $(\alpha,\beta)$-metric, where $\beta$ is a closed and conformal one-form.
\end{abstract}

\textbf{Keywords:} Finsler geometry; general $(\alpha, \beta)$-metrics; Weak Landsberg metric; generalized Unicorn problem.

\section{Introduction}
A Finsler manifold $(M, F)$ is a $C^\infty$ manifold equipped with a Finsler metric which is a continuous function $F:TM\rightarrow [0,\infty)$ with the following properties:\\
(1) Smoothness: $F(x,y)$ is $C^\infty$ on $TM-\lbrace 0\rbrace $.\\
(2)Positive homogeneity: $F(x,\lambda y)=\lambda F(x,y)$ for all $\lambda>0$.\\
(3) Strong convexity: The fundamental tensor $\left(g_{ij}(x,y)\right)$ is positive definite at all $(x,y)\in TM-\lbrace 0\rbrace$, where
\begin{equation*}
g_{ij}(x,y):=\dfrac{1}{2}[F^2]_{y^iy^j} (x,y).
\end{equation*}

There are a lot of non-Riemannian metrics in Finsler geometry. Randers metric is the simplest non-Riemannian Finsler metric, which was introduced by G. Randers in \cite{G.Randers-1941}. As a generalization of Randers metric, $(\alpha, \beta)$-metric is defined by the following form
\begin{equation*}
F=\alpha \phi(s),~~s=\dfrac{\beta}{\alpha},
\end{equation*}
which $\alpha$ is a Riemannian metric, $\beta$ is a one-form and $\phi(s)$ is a $C^\infty$ positive function. A more general metric class called \textit{general $(\alpha, \beta)$-metric} was first introduced by C. Yu and H. Zhu in \cite{Yu-Zhu}. By definition, it is a Finsler metric expressed in the following form
\begin{equation*}
F=\alpha \phi(b^2,s),~~s=\dfrac{\beta}{\alpha}.
\end{equation*}
This class of Finsler metrics include some Finsler metrics constructed by Bryant (see \cite{Bryant-2002}, \cite{Bryant-1997} and \cite{Bryant-1996}).

In Finsler geometry, There are several classes of metrics, such as Berwald metric, Landsberg metric, and weak Landsberg metric.
We know that Berwald metric is a bit more general than Riemannian and Minkowskian metric. However, every Berwald metric is not only a Landsberg metric but also a weak Landsberg metric.
For $(\alpha,\beta)$-metrics, by definitions, we have the following relations\\

$\lbrace$ Riemannian $\rbrace \& \lbrace$ locally Minkowskian $\rbrace \subset \lbrace$ Berwald $\rbrace \subset \lbrace$ Landsberg $\rbrace$,\\
~\\
and\\

$\lbrace$ Landsberg metrics $\rbrace \subset \lbrace$ weak Landsberg metrics $\rbrace$.\\

The pivotal question is, is there a Landsberg metric that is not Berwaldian? This problem was called "unicorn" by D. Bao \cite{Bao-2007}. Morever, is there weak Landsberg metric that is not Berwaldian? This problem was called "generalized unicorn" by Zou and Cheng \cite{Zou-Cheng-2006}.

In 2009, Z. Shen has proved that a regular $(\alpha,\beta)$-metric $F=\alpha \phi(\beta/\alpha)$-metric on $M$ of dimension $n>2$ is a Landsberg metric if and only if $F$ is a Berwald metric \cite{Shen-2009}. On the other hand, Z. Shen and G.S. Asanov have constructed {\it almost regular} $(\alpha,\beta)$-metrics which are Landsberg metrics but not Berwald metrics respectively (see \cite{Asanov-2006} and \cite{Shen-2009}).
In 2014, Y. Zou and Cheng have showed that if $\phi=\phi(s)$ is a polynomial in $s$, then $F$ is a weak Landsberg metric if and only if $F$ is a Berwald metric. They generalized the main theorem on unicorn problem for regular $(\alpha,\beta)$-metrics in \cite{Zou-Cheng-2006}.

In general $(\alpha,\beta)$-metrics, Zohrehvand and Maleki in \cite{Zohrehvand-2016} showed that hunting for an unicorn cannot be successful in the class of metrics where $\beta$ is a closed and conformal one-form, i.e.
$b_{i|j}=ca_{ij}$,
where $b_{i|j}$ is the covariant derivation of $\beta$ with respect to $\alpha$ and $c=c(x)$ is a scalar function on $M$. In this paper, we show that Landsberg metric and weak Landsberg metric are equivalent in the class of general $(\alpha,\beta)$-metrics where $\beta$ is a closed and conformal one-form and thus hunting for an generalized unicorn cannot be successful.
\begin{theorem}\label{Main Theo}
Let $F=\alpha\phi(b^2,\frac{\beta}{\alpha})$ be a non-Riemannian general $(\alpha,\beta)$-metric on an $n$-dimentional manifold $M$ and $\beta$ satisfies
\begin{equation}\label{condition-closed-conformal}
 b_{i|j}=ca_{ij}, 
\end{equation}
where $b_{i|j}$ is the covariant derivation of $\beta$ with respect to $\alpha$ and $c=c(x)$ is a scalar function on $M$. Then $F$ is a weak Landsberg metric if and only if it is Landsberg metric.
\end{theorem}
\begin{corollary}\label{coro1}
Let $F=\alpha\phi(b^2,\frac{\beta}{\alpha})$ be a non-Riemannian general $(\alpha,\beta)$-metric on an $n$-dimentional manifold $M$ and $\beta$ is closed and conformal.
Then $F$ is a weak Landsberg metric if and only if it is Berwald metric.
\end{corollary}
Thus, under the certain condition, the generalized unicorn's problem cannot be successful in the class of general $(\alpha,\beta)$-metrics.
In this case $F$ can be expressed by
\[
F=\alpha\varphi
\left( \dfrac{s^2}{e^{\int{(\frac{1}{b^2}-b^2\theta)db^2}}+s^2\int{\theta e^{\int{(\frac{1}{b^2}-b^2\theta)db^2} }db^2}}
\right) e^{ \int{(\frac{1}{2}b^2\theta-\frac{1}{b^2} ) db^2} }s
\]
where $\varphi(.)$ is any positive continuously differentiable function and $\theta$ is a smooth function of $b^2$ \cite{Zhu-2015}.

\section{Preliminaries}
For a given Finsler metric $F=F(x,y)$, the geodesic of $F$ satisfies the following differential equation:
\[ \dfrac{d^2x^i}{dt^2}+2 G^i\left(x,\dfrac{dx}{dt}\right)=0,  \]
where $G^i=G^i(x,y)$ are called the geodesic coefficients defined by
\[ G^i=\dfrac{1}{4} g^{il} \lbrace [F^2]_{x^my^l}y^m-[F^2]_{x^l} \rbrace. \]

%A finsler metric $F$ on $M$ is called a Berwald metric, if the geodesic coefficients $G^i$'s are quadratic functions of $y$-coordinates in each tangent space.
For a tangent vector $y:=y^i\dfrac{\partial}{\partial x^i}\in T_xM$, the Berwald curvature $\textbf{B}_y:=B_{jkl}^i \frac{\partial}{\partial x^i}\otimes dx^j\otimes dx^k \otimes dx^l$, can be expressed by
\[ B_{jkl}^i:=\dfrac{\partial^3 G^i}{\partial y^j\partial y^k\partial y^l}.\]
Thus, a Finsler metric $F$ is a Berwald metric if and only if $\textbf{B}=0$.

The Landsberg curvature can be expressed by 
\[ L_{jkl}:=-\dfrac{1}{2} y^m g_{im} B_{jkl}^i \] %~~~,~~~\forall j, k, l. 
A Finsler metric $F$ is a Landsberg metric if and only if $L_{jkl}=0$, i.e.
\[ y^m g_{im}B_{jkl}^i =0. \]
Thus,  Berwald metrics are always Landsberg metrics.

There is a weaker non-Riemannian quantity than the Landsberg curvature $\textbf{L}$ in Finsler geometry, $\textbf{J}=J_k dx^k$, where
\begin{equation}
J_k:=g^{ij} L_{ijk},
\end{equation}
and $(g^{ij})=(g_{ij})^{-1}$.
We call $\textbf{J}$ the mean Landsberg curvature of Finsler metric $F$.
A Finsler metric $F$ is called weak Landsberg metric if its mean Landsberg curvature $\textbf{J}$ vanishes \cite{Shen-2001}.

Let $F$ be a Finsler metric on a manifold $M$. $F$ is called a general $(\alpha, \beta)$-metrics, if $F$ can be expressed as the form
\begin{equation}\label{g(ab)-mdef}
 F=\alpha\phi(b^2, s) ~~,~~s=\dfrac{\beta}{\alpha},~~ b^2:=||\beta||_\alpha^2 ,
\end{equation}
where $\alpha$  is a Riemannian metric and $\beta:=b_i(x)y^i$ is an one-form with $||\beta||_\alpha <b_0$ for every $x\in M$.
The function $\phi=\phi(b^2, s)$ is a positive $C^\infty$ function satisfying
\begin{equation*}
\phi-s\phi_2>0,~~\phi-s\phi_2+(b^2-s^2)\phi_{22}>0
\end{equation*}
when $n\geq 3$ or
\begin{equation*}
\phi-s\phi_2+(b^2-s^2)\phi_{22}>0
\end{equation*}
when $n=2$, where $s$ and $b$ are arbitrary numbers with $|s|\leq b <b_0$, for some $0<b_0\leq +\infty$.
In this case, the fundamental tensor is given by \cite{Yu-Zhu}
\begin{equation}\label{def-g}
g_{ij}=\rho a_{ij}+\rho_0 b_ib_j +\rho_1(b_i\alpha_{y^j}+b_j\alpha_{y^i}) -s\rho_1 \alpha_{y^i}\alpha_{y^j},
\end{equation}
where
\begin{equation}\label{def-rho}
\rho:=\phi(\phi-s\phi_2),~~\rho_0:=\phi\phi_{22}+\rho_2\rho_2,~~\rho_1=(\phi-s\phi_2)\phi_2-s\phi\phi_{22}.
\end{equation}
%\[ \rho:=\phi(\phi-s\phi_2),~~\rho_0:=\phi\phi_{22}+\rho_2\rho_2,~~\rho_1=(\phi-s\phi_2)\phi_2-s\phi\phi_{22}. \]
Moreover,
\begin{align}
&\det(g_{ij})=\phi^{n+1}(\phi-s\phi_2)^{n-2}(\phi-s\phi_2+(b^2-s^2)\phi_{22})\det(a_{ij}), \\
&g^{ij}=\dfrac{1}{\rho}\lbrace a^{ij}+\eta b^ib^j +\eta_0 \alpha^{-1}(b^iy^j+b^jy^i)+ \eta_1\alpha^{-2}y^iy^j \rbrace , \label{g^ij}
\end{align}
where
$(g^{ij})=(g_{ij})^{-1},~(a^{ij})=(a_{ij})^{-1},~ b^i=a^{ij}b_j$,
\begin{align}
\eta &=-\dfrac{\phi_{22}}{\phi-s\phi_2+(b^2-s^2)\phi_{22}} ,  \quad
\eta_0 =-\dfrac{(\phi-s\phi_2)\phi_2 -s\phi\phi_{22}}{\phi(\phi-s\phi_2+(b^2-s^2)\phi_{22})} , \label{def-eta} \\
\eta_1 &=\dfrac{(s\phi+(b^2-s^2)\phi_2)((\phi-s\phi_2)\phi_2-s\phi\phi_{22})}{\phi^2(\phi-s\phi_2+(b^2-s^2)\phi_{22})}. \nonumber
\end{align}
Not that, we use the indices 1 and 2 as the derivation with respect to $b^2$ and $s$, respectively \cite{Yu-Zhu}.

Let $b_{i|j}$ denote the coefficients of the covariant derivative of $\beta$ with respect to $\alpha$. Let \cite{Yu-Zhu}
\begin{align*}
r_{ij} &:=\dfrac{1}{2}(b_{i|j}+b_{j|i}),~~ r_00:=r_{ij}y^iy^j,~~ r_i:=b^jr_{ji},~~ r_0:=r_iy^i,~~ r^i:=a^{ij}r_j,~~ r:=b^ir_i \\
s_{ij} &:=\dfrac{1}{2}(b_{i|j}-b_{j|i}),~~ s_0^i:=a^{ik}s_{kj}y^j,~~ s_i:=b^js_{ji},~~ s_0:=s_iy^i,~~ s^i:=a^{ij}s_j .
\end{align*}
$\beta$ is a closed one-form if and only if $s_{ij}=0$, and it is a conformal one-form with respect to $\alpha$ if and only if
$b_{i|j}+b_{j|i}=ca_{ij}$, where $c=c(x)$ is a nonzero scalar function on $M$. Thus, $\beta$ is closed and conformal with respect to $\alpha$ if and only if $b_{i|j}=ca_{ij}$, where $c=c(x)$ is a nonzero scalar function on $M$.

For a general $(\alpha, \beta)$-metric, its spray coefficients $G^i$ are related to the spray coefficients $G_\alpha^i$ of $\alpha$ by \cite{Yu-Zhu}
\begin{align}\label{spray coefficient}
G^i =&G_\alpha^i+\alpha Q s_{~0}^i+\lbrace \Theta (-2\alpha Q s_0+r_00+2\alpha^2 Rr)+\alpha\Omega (r_0+s_0)\rbrace l^i \nonumber \\
&+\lbrace \Psi (-2\alpha Q s_0+r_00 +2\alpha^2 Rr)+\alpha \Pi (r_0+s_0) \rbrace b^i -\alpha^2 R(r^i+s^i),
\end{align}
where $l^i:=\dfrac{y^i}{\alpha}$ and
\begin{align*}
Q&:=\dfrac{\phi_2}{\phi-s\phi_2},~~~~ 
R:=\dfrac{\phi_1}{\phi-s\phi_2},~~~~\\
\Theta &:=\dfrac{(\phi-s\phi_2)\phi_2-s\phi\phi_{22}}{2\phi(\phi-s\phi_2+(b^2-s^2)\phi_{22})},~~~~
\Psi:=\dfrac{\phi_{22}}{2(\phi-s\phi_2+(b^2-s^2)\phi_{22})} \\
\Pi &:=\dfrac{(\phi -s\phi_2)\phi_{12} -s\phi_1 \phi_{22}}{(\phi-s\phi_2)(\phi-s\phi_2+(b^2-s^2)\phi_{22})},~~~~
\Omega:=\dfrac{2\phi_1}{\phi}-\dfrac{s\phi+(b^2-s^2)\phi_2}{\phi}\Pi
\end{align*}

When $\beta$ is closed and conformal one-form, i.e. satisfies
\eqref{condition-closed-conformal},
then
\begin{equation*}
r_{00}=c\alpha^2,~~~ r_0=c\beta,~~~ r=cb^2,~~~ r^i=cb^i,~~~ s_{~0}^i=s_0=s^i=0
\end{equation*}
Substituting this into \eqref{spray coefficient} yields \cite{Zhu-2015}
\begin{equation}\label{subs-spray-coeff}
G^i:=G_\alpha^i +c\alpha^2\lbrace \Theta (1+2Rb^2)+s\Omega \rbrace l^i+ c\alpha^2\lbrace \Psi(1+2Rb^2)+s\Pi -R \rbrace b^i
\end{equation}
If we have
\begin{align}
E &:=\dfrac{\phi_2 +2s\phi_1}{2\phi}-H\dfrac{s\phi +(b^2-s^2)\phi_2}{\phi} \label{def-E} \\
H&:=\dfrac{\phi_{22}-2(\phi_1-s\phi_{12})}{2(\phi-s\phi_2+(b^2-s^2)\phi_{22})}, \label{def-H}
\end{align}
then from \eqref{subs-spray-coeff}
\begin{equation}\label{spray-coeff with cond-closed and conf}
G^i:=G_\alpha^i+c\alpha^2 E l^i+c\alpha^2 Hb^i.
\end{equation}
The Berwald curvature of a general $(\alpha,\beta)$-metric, when $\beta$ is a closed and conformal one-form, is computed in \cite{Zhu-2015}:
\begin{proposition}\label{prop-Berwald}
Let $F=\alpha\phi(b^2,s),~s=\beta/\alpha$, be a general $(\alpha, \beta)$-metric on an $n$-dimensional manifold $M$. Suppose that $\beta$ satisfies \eqref{condition-closed-conformal}, then the Berwald curvature of $F$ is given by
\begin{equation}
B^i_{~jkl} = \dfrac{c}{\alpha} U^i_{~jkl}
\end{equation}
where
\begin{align*}
U^i_{~jkl}:=&\lbrace [(E-sE_2)a_{kl}+E_{22} b_k b_l]\delta^i_j + s(3E_{22}+sE_{222})l_l l_j b_k l^i \nonumber\\
  &-(E_{22}+sE_{222}) b_l l_j b_k l^i   \rbrace (k \rightarrow l \rightarrow j \rightarrow k)\nonumber\\
  &-\lbrace sE_{22}[a_{jl}b_k l^i+(l_k b_l+l_l b_k)\delta^i_j ]\nonumber\\
  &+(E-sE_2-s^2E_{22})(a_{jl}l^i+l_l\delta^i_j)l_k   \rbrace (k \rightarrow l \rightarrow j \rightarrow k)\nonumber\\
  &+\lbrace (3E-3sE_2-6s^2E_{22}-s^3E_{222})l_j l_k l_l + E_{222}b_l b_k b_j \rbrace l^i \nonumber\\
  &+\lbrace (H_2-sH_{22})(b_j-sl_j)a_{kl}-(H_2-sH_{22}-s^2H_{222})b_ll_jl_k \nonumber\\
  &-sH_{222}b_kb_ll_j  \rbrace b^i (k \rightarrow l \rightarrow j \rightarrow k)\nonumber\\
  &+\lbrace s(3H_2-3sH_{22}-s^2H_{222})l_jl_kl_l+H_{222} b_jb_lb_k  \rbrace b^i ,
\end{align*}
where $E$ and $H$ is defined in \eqref{def-E} and \eqref{def-H}, $l_j:=a_{ij}l^i$ and
$(k \rightarrow l \rightarrow j \rightarrow k)$ denotes cyclic permutation.
\end{proposition}
\section{The mean Landsberg curvature of General $(\alpha, \beta)$-metrics}
By use of Proposition \ref{prop-Berwald}, M. Zohrehvand and H. Maleki calculated the Landsberg curvature of a general $(\alpha, \beta)$-metric in \cite{Zohrehvand-2016}, when $\beta$ is a closed and conformal one-form:
\begin{proposition}\label{prop-Landsberg-prop}
Let $F=\alpha\phi(b^2,s),~s=\beta/\alpha$, be a general $(\alpha, \beta)$-metric on an $n$-dimensional manifold $M$. Suppose that $\beta$ satisfies \eqref{condition-closed-conformal}, then the Landsberg curvature of $F$ is given by
\begin{equation}\label{prop-Landsberg}
L_{jkl}=-\dfrac{c}{2}\phi V_{jkl}
\end{equation}
where
\begin{align}
V_{jkl}:=&\lbrace [(E-sE_2)a_{kl}+E_{22} b_k b_l](\phi l_j+\phi_2(b_j-sl_j))  \nonumber\\
  &-[(E_{22}+sE_{222})(b_l-sl_l)-2E_{22}sl_l]\phi b_k l_j  \rbrace   (k \rightarrow l \rightarrow j \rightarrow k)     \nonumber\\
  &-\lbrace sE_{22}[a_{jl}b_k \phi +(l_lb_k+l_kb_l)(\phi l_j+\phi_2(b_j-sl_j))] \nonumber\\
  &+(E-sE_2-s^2E_{22})[a_{jl}l_k\phi+(\phi l_j+\phi_2(b_j-sl_j))l_kl_l] \rbrace (k \rightarrow l \rightarrow j \rightarrow k)\nonumber\\
  &+[(3E-3sE_2-6s^2E_22-s^3E_{222})l_j l_k l_l + E_{222}b_j b_k b_l]\phi   \nonumber\\
  &+\lbrace (H_2-sH_{22})(a_{kl}(b_j-sl_j)-b_l l_k l_j)  \nonumber\\
  &-sH_{222}l_jb_l(b_k-sl_k) \rbrace (s\phi +(b^2-s^2)\phi_2)  (k \rightarrow l \rightarrow j \rightarrow k)\nonumber\\
  &+[s(3H_2-3sH_{22}-s^2H_{222})l_jl_kl_l+H_{222} b_jb_kb_l](s\phi +(b^2-s^2)\phi_2).
\end{align}
%where $(k \rightarrow l \rightarrow j \rightarrow k)$ denotes cyclic permutation.
\end{proposition}

By use of Proposition \ref{prop-Landsberg-prop} and Maple, we can calculate the mean Landsberg curvature of a general $(\alpha, \beta)$-metric, when $\beta$ is a closed and conformal one-form:
\begin{proposition}\label{prop-meanLandsberg-prop}
Let $F=\alpha\phi(b^2,s),~s=\beta/\alpha$, be a general $(\alpha, \beta)$-metric on an $n$-dimensional manifold $M$. Suppose that $\beta$ satisfies \eqref{condition-closed-conformal}, then the Landsberg curvature of $F$ is given by
\begin{equation}\label{prop-mean Landsberg}
J_j=-\dfrac{c\phi}{2\rho} W_{j},
\end{equation}
where
\begin{align}\label{W_j}
W_j:=& \lbrace (E-sE_2)(n+1)\phi_2 +3E_{22} \phi_2(b^2-s^2)-sE_{22}(n+1) \phi+E_{222}\phi(b^2-s^2)  \nonumber\\
&\quad +\lbrace (H_2-sH_{22})(n+1)+H_{222}(b^2-s^2)\rbrace (s\phi +(b^2-s^2)\phi_2)  \nonumber\\
&\quad+3\eta (E-sE_2)\phi_2(b^2-s^2) +3\eta E_{22}\phi_2(b^2-s^2)^2 -3s\eta E_{22}\phi(b^2-s^2) \nonumber\\
&\quad +\eta E_{222}(b^2-s^2)^2 \phi +\eta [ 3(H_2-sH_{22})(b^2- s^2)+H_{222}(b^2-s^2)^2 ]\nonumber\\
&~~~\times (s\phi +(b^2-s^2)\phi_2) \rbrace (b_j-sl_j)
\end{align}
where $\rho$ and $\eta$ is defined in \eqref{def-rho} and \eqref{def-eta} and $l_j:=a_{ij}l^i$.
\end{proposition}
\begin{proof}
To compute the mean Landsberg curvature $J_j:=g^{kl}L_{jkl}$ using \eqref{prop-Landsberg}, we need to \eqref{g^ij}. Using
$L_{jkl}y^k=0$, we get
\begin{equation}\label{proof-J_j}
J_j:=\dfrac{1}{\rho}\lbrace a^{kl}+\eta b^kb^l\rbrace L_{jkl}=\dfrac{c\phi}{2\rho} \lbrace a^{kl}+\eta b^kb^l\rbrace V_{jkl}.
\end{equation}
We need
\begin{align*}
a^{kl}V_{jkl}&=\lbrace (E-sE_2)(n+1)\phi_2 +3E_{22} \phi_2(b^2-s^2)-sE_{22}(n+1) \phi+E_{222}\phi(b^2-s^2)  \\
&~~+[ (H_2-sH_{22})(n+1)+H_{222}(b^2-s^2)] (s\phi +(b^2-s^2)\phi_2) \rbrace (b_j-sl_j),
\end{align*}
and
\begin{align*}
\eta b^kb^l V_{jkl}&=\lbrace 3\eta (E-sE_2)\phi_2(b^2-s^2) +3\eta E_{22}\phi_2(b^2-s^2)^2-3s\eta E_{22}\phi(b^2-s^2) \\
&\quad + \eta E_{222}(b^2-s^2)^2 \phi +\eta [ 3(H_2-sH_{22})(b^2- s^2)+H_{222}(b^2-s^2)^2 ] \\
&~~~\times (s\phi +(b^2-s^2)\phi_2) \rbrace (b_j-sl_j).
\end{align*}
Substituting these in \eqref{proof-J_j}, we obtain \eqref{prop-mean Landsberg}.\qed
\end{proof}
Now, we can obtain the necessary and sufficient conditions for a general $(\alpha, \beta)$-metric to be weak Landsbergian.
\begin{proposition}\label{conditon-mean-Landsberg}
Let $F=\alpha\phi(b^2,s),~s=\beta/\alpha$, be a general $(\alpha, \beta)$-metric on an $n$-dimensional manifold $M$. Suppose that $\beta$ satisfies \eqref{condition-closed-conformal}, then $F$ is weak Landsberg metric if and only if the following equations hold:
\begin{align}
&E_{22}=0 ,~~ H_{222}=0, \label{equation of weak Landsberg-1}\\
&(E-sE_2)\phi_2 +(H_2-sH_{22})(s\phi +(b^2-s^2)\phi_2 =0 \label{equation of weak Landsberg-2}
\end{align}
\end{proposition}
\begin{proof}
Let $F=\alpha\phi(b^2,s),~s=\beta/\alpha$, be a weak Landsberg metric, where $\beta$ is a closed and conformal one-form. From Proposition \ref{prop-Landsberg-prop}, it concluded 
\begin{equation}\label{W_1}
W_j=0, 
\end{equation}
where $W_j$ is defined in \eqref{W_j}.
%\begin{align}\label{W_1}
%& (E-sE_2)(n+1)\phi_2 +3E_{22} \phi_2(b^2-s^2)-sE_{22}(n+1) %\phi+E_{222}\phi(b^2-s^2)  \nonumber\\
%&+\lbrace (H_2-sH_{22})(n+1)+H_{222}(b^2-s^2)\rbrace (s\phi +(b^2-s^2)\phi_2)  \nonumber\\
%&-3\eta (E-sE_2)\phi_2(b^2-s^2) -3\eta E_{22}\phi_2(b^2-s^2)^2+3s\eta E_{22}\phi(b^2-s^2) \nonumber\\
%&- \eta E_{222}(b^2-s^2)^2 \phi -\eta [ 3(H_2-sH_{22})(b^2- s^2)+H_{222}(b^2-s^2)^2 ]\nonumber\\
%&\times (s\phi +(b^2-s^2)\phi_2) =0.
%\end{align}
We can rewrite \eqref{W_1} as following:
\begin{align}
&[1+n+3(b^2-s^2)\eta ][(E-sE_2)\phi_2 +(H_2-sH_{22})(s\phi +(b^2-s^2)\phi_2)]  \nonumber\\
&+(b^2-s^2)[1+(b^2-s^2)\eta][s\phi+(b^2-s^2)\phi_2 ] H_{222} \nonumber\\
&+\lbrace 3(b^2-s^2)[1+(b^2-s^2)\eta]\phi_2-[1+n+3(b^2-s^2)\eta] s\phi\rbrace E_{22} \nonumber\\
&+(b^2-s^2)[1+(b^2-s^2)\eta]\phi E_{222}=0.
\end{align}
Thus
\begin{align}
&[1+n+3(b^2-s^2)\eta ][(E-sE_2)\phi_2 +(H_2-sH_{22})(s\phi +(b^2-s^2)\phi_2)]=0 \label{eq2}\\
&(b^2-s^2)[1+(b^2-s^2)\eta][s\phi+(b^2-s^2)\phi_2 ] H_{222}=0 \label{eq1-2}\\
&\lbrace 3(b^2-s^2)[1+(b^2-s^2)\eta]\phi_2-[1+n+3(b^2-s^2)\eta]s \phi\rbrace E_{22}=0 \label{eq1-1}\\
&(b^2-s^2)[1+(b^2-s^2)\eta]\phi E_{222}=0.
\end{align}
Since $1+n+3(b^2-s^2)\eta \neq 0$, we have from \eqref{eq2}
\begin{equation}
(E-sE_2)\phi_2 +(H_2-sH_{22})(s\phi +(b^2-s^2)\phi_2)=0. 
\end{equation}
From \eqref{eq1-2}, we know $(b^2-s^2)[1+(b^2-s^2)\eta]\neq 0$, then $s\phi+(b^2-s^2)\phi_2=0$ or $H_{222}=0$. If $s\phi+(b^2-s^2)\phi_2=0$, we have
\begin{equation*}
\phi=\sigma(b^2)\sqrt{s^2-b^2}
\end{equation*}
where $\sigma(b^2)$ is any positive smooth function. This is a Riemannian case and we have
\begin{equation}
H_{222}=0.
\end{equation} 
In \eqref{eq1-1}, we have 
$3(b^2-s^2)[1+(b^2-s^2)\eta]\phi_2-[1+n+3(b^2-s^2)\eta]s\phi\neq 0$ and 
\begin{equation}
E_{22}=0.
\end{equation}\qed
\end{proof}

\textbf{Proof of Theorem \ref{Main Theo}.}Since in \cite{Zohrehvand-2016} is obtained exactly \eqref{equation of weak Landsberg-1} and \eqref{equation of weak Landsberg-2} for Landsberg metrics, according to \ref{conditon-mean-Landsberg}, it is obvious. \qed

\textbf{Proof of Corollary \ref{coro1}.}
Here, by \cite{Zohrehvand-2016}, we have $E-sE_2=0$ and $H_2-sH_{22}=0$ also.
In \cite{Zhu-2015}, it is proved that a general $(\alpha, \beta)$-metric where $\beta$ is closed and conformal one-form, is a Berwald metric if and only if
\[ E-sE_2=0,~~H_2-sH_{22}=0 \]
and in this case $F$ can be expressed by
\begin{equation}\label{General formula}
F=\alpha\varphi
\left( \dfrac{s^2}{e^{\int{(\frac{1}{b^2}-b^2\theta)db^2}}+s^2\int{\theta e^{\int{(\frac{1}{b^2}-b^2\theta)db^2} }db^2}}
\right) e^{ \int{(\frac{1}{2}b^2\theta-\frac{1}{b^2} ) db^2} }s
\end{equation}
where $\varphi(.)$ is any positive continuously differentiable function and $\theta$ is a smooth function of $b^2$ \cite{Zhu-2015}. This complete the proof.\qed
\section{Examples}
In this section, we will explicitly construct some new examples. \\
\begin{example}
 Take $\varphi(t):=1+\xi t$ and $\theta(b^2):=1$, then by \eqref{General formula}
\[ \phi(b^2,s)=\dfrac{\left(s^2(\xi e^{\frac{b^4}{2}}-1)+b^2\right)e^{\frac{b^4}{4}}s}{b^2(b^2-s^2)}. \]
We can see that $\phi(b^2,s)$ satisfies in \eqref{equation of weak Landsberg-1} and \eqref{equation of weak Landsberg-2}. Moreover, the corresponding general $(\alpha,\beta)$-metrics
\[ F:=\alpha\phi\left(b^2, \frac{\beta}{\alpha}\right)\]
are Landsberg and weak Landsberg metrics, i.e. $L_{jkl}=0$ and $J_j=0$.
\end{example}
\begin{example}
Take $\varphi(t):=\frac{\mu}{1+\xi t}$ and $\theta(b^2):=\frac{\varepsilon}{1+\xi b^2}$ in \eqref{General formula}, we have
\[
\phi(b^2,s):=\dfrac{\mu (b^2-s^2)e^{\frac{1}{2}\frac{\varepsilon\left(\xi b^2-\ln(\xi b^2+1) \right)}{\xi^2}}s}{\left(s^2b^2\xi e^{\frac{b^2\varepsilon\xi-\xi^2 \ln(b^2)-\varepsilon\ln(\xi b^2+1) }{\xi^2}}+b^2-s^2\right)b^2} .
\]
The corresponding general $(\alpha,\beta)$-metrics are Landsberg and weak Landsberg metrics.
\end{example}
\section*{Acknowledgment}
The authors would like to express their special thanks to Dr. M. Zohrehvand for his valuable opinions on this paper.

%\section*{References}

%\address{
Department of Mathematics\\
 Faculty of Mathematics Science\\
 University of Mazandaran, Babolsar, Iran\\
 ali.ala@stu.umz.ac.ir \\
behzadi@umz.ac.ir\\
rafiei-rad@umz.ac.ir
%}
\end{document}